\documentclass[12pt, a4paper]{article}
\usepackage[utf8]{inputenc}

\usepackage[margin = 1in]{geometry}
\usepackage{amsfonts,amsmath,amssymb,amsthm}
\usepackage{graphicx,subcaption}
\usepackage{booktabs}
\usepackage{lipsum}
\usepackage{hyperref}
\usepackage{multicol}
\hypersetup{
   colorlinks = true,
   citecolor = blue,
   urlcolor = purple,
   linkcolor = red
}

\usepackage{mathrsfs}
\usepackage{multirow}
\usepackage[makeroom]{cancel}
\usepackage{tikz}
\usepackage{framed,color}
\definecolor{shadecolor}{rgb}{1,0.8,0.3}
\usepackage{fancybox}
\usepackage{caption}
\usepackage{algorithm,algorithmic} 
\usepackage{multicol}
\usepackage{aliascnt}
\usepackage{setspace}
\usepackage{longtable}
\allowdisplaybreaks

\usetikzlibrary{decorations.markings,arrows}

\allowdisplaybreaks

\title{\textbf{Asymptotics in Multiple Hypotheses Testing under Dependence: beyond Normality}}

\date{}
\vspace{5mm}
\author{Monitirtha Dey \\
        Institute for Statistics, University of Bremen, 28359 Bremen, Germany \\
        \textit{\href{mdey@uni-bremen.de}{mdey@uni-bremen.de}, 
        \href{monitirthadey3@gmail.com}{monitirthadey3@gmail.com}}}

\usepackage{natbib}
\setcitestyle{authoryear, open={(},close={)}}

\begin{document}

\maketitle
%\vspace{3mm}
\theoremstyle{plain}
\newtheorem{axiom}{Axiom}
\newtheorem{remark}{Remark}
\newtheorem{corollary}{Corollary}[section]
\newtheorem{claim}[axiom]{Claim}
\newtheorem{theorem}{Theorem}[section]
\newtheorem{lemma}{Lemma}[section]

\newaliascnt{lemmaa}{theorem}
\newtheorem{lemmaa}[lemmaa]{Theorem}
\aliascntresetthe{lemmaa}
\providecommand*{\lemmaaautorefname}{Theorem}

% \documentclass[times,sort&compress,3p]{elsarticle}
% \journal{Journal of Multivariate Analysis}
% \usepackage[labelfont=bf]{caption}
% \renewcommand{\figurename}{Fig.}

%\usepackage{amsmath,amsfonts,amssymb,amsthm,booktabs,color,epsfig,graphicx,url}
% \hyperef was included in the above line. i have changed it in line 22.

\theoremstyle{plain}% Theorem-like structures provided by amsthm.sty
\newtheorem{exa}{Example}
\newtheorem{rem}{Remark}
\newtheorem{proposition}{Proposition}

\theoremstyle{definition}
\newtheorem{definition}{Definition}
\newtheorem{example}{Example}

% \RequirePackage[colorlinks,citecolor=blue,urlcolor=blue]{hyperref}

% \usepackage{lineno}

% \begin{document}
% \linenumbers
% \begin{frontmatter}

% \title{On Asymptotic Behaviour of Stepwise Multiple Testing Procedures}
% \author{Monitirtha Dey }
%\author[2]{Author Two}

% \address{Indian Statistical Institute, Kolkata}
% \address{Email address: monitirtha.d\_r@isical.ac.in}
%\address[2]{Address of Author Two in his country's language and rules}

%\cortext[mycorrespondingauthor]{Corresponding author. Email address: \url{monitirthadey3@gmail.com}}

\begin{abstract}
Correlated observations are ubiquitous phenomena in a plethora of scientific avenues. Tackling this dependence among test statistics has been one of the pertinent problems in simultaneous inference. However, very little literature exists that elucidates the effect of correlation on different testing procedures under general distributional assumptions. In this work, we address this gap in a unified way by considering the multiple testing problem under a general correlated framework. We establish an upper bound on the family-wise error rate(FWER) of Bonferroni's procedure for equicorrelated test statistics. Consequently, we find that for a quite general class of distributions, Bonferroni FWER asymptotically tends to zero when the number of hypotheses approaches infinity. We extend this result to general positively correlated elliptically contoured setups. We also present examples of distributions for which Bonferroni FWER has a strictly positive limit under equicorrelation. We extend the limiting zero results to the class of step-down procedures under quite general correlated setups. Specifically, the probability of rejecting at least one hypothesis approaches zero asymptotically for any step-down procedure. The results obtained in this work generalize existing results for correlated Normal test statistics and facilitate new insights into the performances of multiple testing procedures under dependence.

\end{abstract}
\vspace{5mm}
\noindent \textbf{\textit{Keywords.}} %alphabetical order
Familywise error rate,
Multiple testing under dependence,
Stepwise Procedures,
Holm's Method,
Elliptically Contoured Distributions.
%Hochberg Procedure \sep
%Hommel's Procedure.
%A necessary keyword \sep
%My favorite keyword \sep
%Our last keyword.

\vspace{2mm}

\newpage
% \tableofcontents
\section{Introduction\label{sec:1}}

Dependence among observations is a natural phenomenon in a plethora of scientific avenues. It is more so when one is interested in asking several related questions. Thus, capturing and tackling the correlation among the test statistics corresponding to the various hypotheses is a problem of pivotal importance in simultaneous inference \citep{Chi2024, Gasparin2024, Sarkar2023, Xu2024}. Hence, the research on how correlation among test statistics affects the performances of multiple testing procedures has garnered much interest lately.  

\vspace{3mm}
The equicorrelation structure portrays the simplest kind of dependence. Although simple in form, the equicorrelated setup arises in several applications, e.g., in the problem of comparing a control against several treatments. Consequently, numerous recent works in multiple testing consider the equicorrelated setup \citep{Delattre, Proschan, royspl}. We note that equicorrelated normal test statistics with common mean and variance are exchangeable. As also noted by \cite{Gasparin2024}, exchangeability can be induced under any configuration by processing the (potentially non-exchangeable) sequence of test statistics in a uniformly random order. Hence, if randomization is allowed, it is always possible to satisfy the
exchangeability assumption even if the starting sequence has an arbitrary dependence.

\vspace{3mm}
\cite{das_2021} have shown that the Bonferroni familywise error rate (FWER henceforth) is asymptotically a convex function in the common correlation $\rho$ under the equicorrelated normal framework. This result, in turn, gives that the Bonferroni FWER is bounded above by $\alpha(1-\rho)$, $\alpha$ being the desired level. To improve this result, \cite{deybhandari} have elucidated that the Bonferroni FWER asymptotically goes to zero for any strictly positive $\rho$. \cite{deystpa} has recently extended this limiting zero result to the general class of step-down FWER-controlling procedures. In a related but different direction, \cite{deycstm} has obtained upper bounds on the Bonferroni FWER in the equicorrelated setup with small and moderate dimensions. However, the bounds presented in \cite{deycstm} do not converge to zero as the number of hypotheses becomes arbitrarily large. Hence, there is a scope to find upper bounds on Bonferroni FWER that approaches zero as the number of hypotheses diverges to infinity. Also, all the works mentioned earlier consider equicorrelated normal distributions. 

\vspace{3mm}
However, very little literature exists that elucidates the effect of correlation on different testing procedures under general distributional assumptions. In this work, we address this gap unifyingly by considering the multiple testing problem under a general correlated framework. We establish an upper bound on the family-wise error rate(FWER) of Bonferroni's procedure for equicorrelated test statistics. Consequently, we find that for a quite general class of distributions, Bonferroni FWER asymptotically tends to zero when the number of hypotheses approaches infinity. We extend this result to general positively correlated elliptically contoured setups. We also present examples of distributions for which Bonferroni FWER has a strictly positive limit under equicorrelation. We extend the limiting zero results to the class of step-down procedures under quite general correlated setups. Specifically, the probability of rejecting at least one hypothesis approaches zero asymptotically for any step-down procedure. The results obtained in this work generalize existing results for correlated Normal test statistics and facilitate new insights into the performances of multiple testing procedures under dependence. 

\vspace{3mm}

This paper is organized as follows. We first formally introduce the framework with the assumptions and summarize some results on the limiting behavior of the Bonferroni method under the equicorrelated normal setup in the next section. Section \ref{sec:3} studies in detail the limiting behaviors of the Bonferroni FWER under general distributional assumptions. Section \ref{sec:4} extends these results to general positively correlated elliptically contoured setups. We extend the limiting zero results to the class of step-down FWER-controlling procedures in Section \ref{sec:5}. A detailed empirical  study is presented in Section \ref{sec:6}. We sketch out our contributions and briefly discuss potential problems in similar avenues in Section \ref{sec:7}.

\vspace{3mm}
Throughout this paper, $H(\mu, \sigma^2)$ represents a cdf $H$  parametrized by the mean $\mu$ and variance $\sigma^2$. 

\section{Preliminaries\label{sec:2}} 
\subsection{The General Equicorrelated Sequence Model Framework}  
We consider a \textit{sequence model} framework \citep{Finner2001a, FDR2007, Proschan}. We have $n$ (possibly correlated) observations from a distribution $F^{\star}$: 
$$X_{i} {\sim} F^{\star}(\mu_{i},1),  \quad i \in \{1, \ldots,n\}.$$
So, $X_{i}$ has mean $\mu_{i}$ and variance $1$. We take unit variances since the literature on the asymptotic multiple testing theory usually considers known variances \citep{Abramovich, Bogdan, Donoho}. 

Here the $X_{i}$’s are equicorrelated: for some $\rho \in (0,1)$, we have 
$$X_{i} = \sqrt{1-\rho}Z_{i} + \sqrt{\rho} U,  \quad i \in \{1, \ldots,n\},$$
where $Z_{i} \sim F(\frac{\mu_{i}}{\sqrt{1-\rho}},1)$ and $U \sim G(0,1)$ independently. Also, $Z_{i}$’s are independent to each other. These kind of equicorrelated setups are routinely considered in multiple testing literature \citep{Delattre, DickhausThesis, Proschan, royspl}. The equicorrelated setup also includes the problem of comparing several treatments against a control. 

\subsection{The Multiple Testing Problem}

This paper considers the following multiple testing problem: 
$$H_{0i}:\mu_{i}=0 \quad vs \quad H_{1i}:\mu_{i}>0, \quad 1 \leq i \leq n.$$ Under the global null $ H_{0}:=\bigcap_{i=1}^{n} H_{0 i}$, each $\mu_{i}$ is zero. Let $\mathcal{A}$ denote the set of indices corresponding to the true nulls. So, under the global null, $\mathcal{A}$ is $\{1, \ldots, n\}$. Throughout this work, $f^{\star}, f, g$ denote the density functions of $F^{\star}, F, G$ respectively (their existence is assumed). Also, $\alpha \in (0,1)$ denotes the desired level of FWER control.

Similar to \cite{deystpa}, let $R_{n}(T)$ and $V_n(T)$ respectively denote the number of rejections and the number of type I errors of a multiple testing procedure (MTP henceforth) $T$. So, $S_n(T)= R_n(T)-V_n(T)$ is the number of correct rejections. The FWER of procedure $T$ is given by
\begin{equation}
    FWER_{T}(n, \alpha, \Sigma_n)= \mathbb{P}_{\Sigma_{n}}(V_{n}(T)\geq 1)
    \label{defFWER}
\end{equation}
where $\Sigma_{n}$ is the covariance matrix of $(X_1, \ldots, X_n)$. 
In many of our results, we shall consider the probability in the r.h.s of \eqref{defFWER} under the global null $H_{0}$ at first (and view that as the definition of FWER) for the technical simplicity. Then we generalize the results obtained in the global null case to arbitrary combination of true and false null hypotheses.

% The present work studies the limiting behaviors of $FWER_T$ for $T$ belonging to a broad class of MTPs under two dependent setups:
% \begin{enumerate}
% \item The equicorrelated setup:

% The covariance matrix of $(X_1, \ldots, X_n)$ is $\Gamma_n$. $\Gamma_n$ has each off-diagonal entry equal to $\rho \geq 0$. 

% \item The non-negatively correlated setup:

% The covariance matrix of $(X_1, \ldots, X_n)$ is $\Sigma_n$. $\Sigma_n$ has non-negative off-diagonal entries. We denote the $(i,j)$'th entry of $\Sigma_n$ as $\rho_{ij}$.
% \end{enumerate}
% The equicorrelated setup \citep{Cohen2009, das_2021, deybhandari, deycstm, Finner2001a, FDR2007, royspl} is the intraclass covariance matrix model, characterizing the exchangeable situation. Although this is a special case of the second one, we are considering them separately since the proof of the result in the general case is based on the corresponding results in the equicorrelated case. 
%  The equicorrelated setup also encompasses the problem of comparing a control against several treatments. However, many scientific disciplines involve variables with more  complex dependence structure (e.g., fMRI studies). These complex dependence scenarios need to be tackled with more general covariance matrices \citep{deybhandari, deycstm}. The second setup also includes the successive correlation covariance matrix, which covers change point problems \citep{Cohen2009}. 

There are several notions of power in simultaneous statistical inference \citep{Dudoit}. One of these is \textit{AnyPwr} \citep{Dudoit}, which is the probability of rejecting at least one false null. So, for a MTP $T$,
 $$AnyPwr_{T} = \mathbb{P}_{\Sigma_{n}}(S_n(T) \geq 1).$$
 
As in \cite{deystpa}, throughout this work, $\Gamma_n$ denotes the $n \times n$ matrix with each diagonal entry equal to 1 and each off-diagonal entry equal to $\rho$. Also, $\Sigma_{n}$ is the $n \times n$ correlation matrix with $(i,j)$'th entry $\rho_{ij}$ for $i \neq j$. 

\subsection{The Bonferroni Procedure}
The Bonferroni procedure has been a cornerstone of simultaneous statistical inference. In one-sided settings (as ours), it rejects $H_{0i}$ if $X_{i}>(F^{\star})^{-1}(1-\alpha/n) (=c_{Bon}, \text{say})$. The Bonferroni FWER (when the test statistics have covariance $\Sigma_n$) is given by
\begin{equation*}
    FWER_{Bon}(n, \alpha, \Sigma_n) =\mathbb{P}_{\Sigma_n}\left(X_{i}>c_{Bon}\right. \text{for some} \left. i \in \mathcal{A}\right)
=\mathbb{P}_{\Sigma_n}\bigg(\bigcup_{i \in \mathcal{A}}\{X_i > c_{Bon}\}\bigg). %\label{defFWER2}
\end{equation*}
We denote $FWER_{T}(n,\alpha, \Gamma_n)$ as $FWER_{T}(n,\alpha, \rho)$ for simplicity. ~\cite{das_2021} show the following in the equicorrelated Normal (i.e, $F^{\star}=F=G$ is the cdf of standard normal) case:
\begin{theorem}\label{thm2.1} \citep{das_2021} Consider the equicorrelated Normal setup. 
Given any $\alpha \in (0,1)$ and $\rho \in [0,1]$, $FWER_{Bon}(n, \alpha, \rho)$ is asymptotically bounded above by $\alpha(1-\rho)$ under the global null.
\end{theorem}
% \newpage
\cite{deybhandari} have obtained the following improvement: 
\begin{theorem}\label{thm2.2} \citep{deybhandari} Consider the equicorrelated Normal setup. Given any $\alpha \in (0,1)$ and $\rho \in (0,1]$, we have
$$\lim_{n \to \infty} FWER_{Bon}(n, \alpha, \rho) = 0$$
under any configuration of true and false null hypotheses.
\end{theorem}

They extend Theorem \ref{thm2.2} to general correlated normal setups:
\begin{theorem}\label{thm2.3} \citep{deybhandari}
Let $\Sigma_n$ be the correlation matrix of $X_1, \ldots, X_n$ with $(i,j)$’th entry $\rho_{ij}$ such that $\liminf \rho_{ij}=\delta>0$. Then, for any $\alpha \in (0,1)$, we have
$$\lim_{n \to \infty}FWER_{Bon}(n,\alpha,\Sigma_{n}) = 0$$
under any configuration of true and false null hypotheses.
\end{theorem}
% Theorem \ref{thm2.3}, a much stronger result than Theorem \ref{thm2.1}, highlights the fundamental problem of using Bonferroni method in a
% simultaneous testing problem. \cite{deybhandari} establish Theorem \ref{thm2.3} using a famous inequality due to Slepian: 
% \begin{theorem}\label{thm2.4} \citep{Slepian}
% Let $\mathbf{X}$ follow $N_{k}(\mathbf{0}, \Sigma)$, where $\Sigma$ is a $k \times k$ correlation matrix. Let $\mathbf{a}=\left(a_{1}, \ldots, a_{k}\right)^{\prime}$ be an arbitrary but fixed real 
% vector. Consider the quadrant probability
% $$g(k, \mathbf{a}, \Sigma)=\mathbb{P}_{\Sigma}\left[\bigcap_{i=1}^{k}\left\{X_{i} \leqslant a_{i}\right\}\right].$$ Let $R=\left(\rho_{i j}\right)$ and $T=\left(\tau_{i j}\right)$ be two positive semidefinite correlation matrices. If $\rho_{i j} \geqslant \tau_{i j}$ holds for all $i, j$, then $g(k, \mathbf{a}, R) \geq g(k, \mathbf{a}, T)$, i.e
% $$
% \mathbb{P}_{R}\left[\bigcap_{i=1}^{k}\left\{X_{i} \leqslant a_{i}\right\}\right] \geqslant \mathbb{P}_{T}\left[\bigcap_{i=1}^{k}\left\{X_{i} \leqslant a_{i}\right\}\right]
% $$
% holds for all $\mathbf{a}=\left(a_{1}, \ldots, a_{k}\right)^{\prime} .$ Moreover, the inequality is strict if $R, T$ are positive definite and if the strict inequality $\rho_{i j}>\tau_{i j}$ holds for some $i, j$.
% \end{theorem} 

In a related but different direction, \cite{deycstm} has obtained upper bounds on the Bonferroni FWER in the equicorrelated setup with small and moderate dimensions. However, those upper bounds do not converge to zero as the number of hypotheses grows to infinity. Also, each of the works and results mentioned in this subsection assumes multivariate normality as the joint distribution of the test statistics. Hence, studying the corresponding limiting behaviors in general equicorrelated structures becomes interesting. In the subsequent sections, we address this gap unifyingly by considering a general equicorrelated setup.

% \cite{deybhandari} showed the following regarding the asymptotic power of Bonferroni's method:

% \begin{theorem}\label{thm6.1} \citep{deybhandari}
% Consider the equicorrelated normal setup with equicorrelation $\rho \in (0,1)$. Suppose $\sup \mu_{i}$ is finite. Then, for any $\alpha \in (0,1)$, $AnyPwr_{Bon}$ goes to zero as $n \to \infty$.
% \end{theorem}

% Throughout this work, $P_{i}$ denotes the $p$-value corresponding to the $i$-th null hypothesis $H_{0i}$, $1 \leq i \leq n$. Also, let $P_{(1)} \leqslant \ldots \leqslant P_{(n)}$ be the ordered $p$-values. Let the null hypothesis corresponding to the p-value $P_{(i)}$ be denoted as $H_{(0i)}$, $1 \leq i \leq n$. 

\section{Limiting Bonferroni FWER under General Equicorrelation \label{sec:3}}
\subsection{Zero Limit of Bonferroni FWER}
Firstly we obtain an upper bound on Bonferroni FWER under the general equicorrelated setup. 
\begin{theorem}\label{thm3.1new} Consider the equicorrelated setup mentioned in Section 2.1 with $\alpha \in (0,1), \rho \in (0,1)$. We have, for each $d \in (0,1)$,
$$
FWER_{Bon}(n, \alpha, \rho) \leqslant 1 - G(0) \cdot F^{n}\left(\frac{c_{Bon}}{\sqrt{1-\rho}}\right)-\left[G\left(c_{Bon} \cdot \frac{1-d}{\sqrt{\rho}}\right)-G(0)\right] \cdot F^{n}\left(\frac{c_{Bon} \cdot d}{\sqrt{1-\rho}}\right)
$$
under the global null hypothesis. 
\end{theorem}

\begin{proof}[\textbf{\upshape Proof of Theorem \ref{thm3.1new}.}]
Under the global null, we have
$$
\begin{aligned}
& FWER_{Bon}(n, \alpha, \rho) =\mathbb{P}_{F^{\star}}\bigg[X_{i} \geqslant c_{Bon} \quad \text {for at least one $i$}\bigg].
\end{aligned}
$$
% \newpage
\noindent Hence, 
\begingroup
\allowdisplaybreaks
\begin{align*}
& 1 - FWER_{Bon}(n, \alpha, \rho) \\
 = & \mathbb{P}_{F^{*}}\left(X_{i} \leqslant c_{Bon} \quad \forall i\right)\\
= & \mathbb{P}_{F^{*}}\left(\sqrt{1-\rho}Z_i + \sqrt{\rho}U \leqslant c_{Bon} \quad \forall i\right)\\
= & \mathbb{P}_{F^{*}}\left(Z_i \leqslant \frac{c_{Bon}-\sqrt{\rho}U}{\sqrt{1-\rho}} \quad \forall i\right)\\
= & \mathbb{E}_{U}\left[F^{n}\left(\frac{c_{Bon}-\sqrt{\rho} U}{\sqrt{1-\rho}}\right)\right] \\
= & \int_{-\infty}^{\infty} F^{n}\left(\frac{c_{Bon}-\sqrt{\rho} u}{\sqrt{1-\rho}}\right) g(u) d u \\
= & \int_{-\infty}^{0} F^{n}\left(\frac{c_{Bon}-\sqrt{\rho} u}{\sqrt{1-\rho}}\right) g(u) d u+\int_{0}^{\infty} F^{n}\left(\frac{c_{Bon}-\sqrt{\rho} u}{\sqrt{1-\rho}}\right) g(u) d u \\
\geqslant & \int_{-\infty}^{0} F^{n}\left(\frac{c_{Bon}}{\sqrt{1-\rho}}\right) g(u) d u+\int_{0}^{\infty} F^{n}\left(\frac{c_{Bon}-\sqrt{\rho} u}{\sqrt{1-\rho}}\right) g(u) d u \\
= & G(0) \cdot F^{n}\left(\frac{c_{Bon}}{\sqrt{1-\rho}}\right)+\int_{0}^{\infty} F^{n}\left(\frac{c_{Bon}-\sqrt{\rho} u}{\sqrt{1-\rho}}\right) g(u) d u \\
\geqslant & G(0) \cdot F^{n}\left(\frac{c_{Bon}}{\sqrt{1-\rho}}\right)+\int_{0}^{c_{Bon} \cdot \frac{(1-d)}{\sqrt{\rho}}} F^{n}\left(\frac{c_{Bon}-\sqrt{\rho} u}{\sqrt{1-\rho}}\right) g(u) d u \quad \text{(for any $d\in (0,1)$)}\\
\geqslant & G(0) \cdot F^{n}\left(\frac{c_{Bon}}{\sqrt{1-\rho}}\right)+\int_{0}^{c_{Bon} \cdot \frac{(1-\alpha)}{\sqrt{\rho}}} F^{n}\left(\frac{c_{Bon} \cdot d}{\sqrt{1-\rho}}\right) g(u) d u \\
= & G(0) \cdot F^{n}\left(\frac{c_{Bon}}{\sqrt{1-\rho}}\right)+\left[G\left(c_{Bon} \cdot \frac{1-d}{\sqrt{\rho}}\right)-G(0)\right] F^{n}\left(\frac{c_{Bon}\cdot d}{\sqrt{1-\rho}}\right).
\end{align*}
\endgroup
The rest follows.
\end{proof}

\begin{remark}
\autoref{thm3.1new} does not require any extra assumption on the densities $f$ or $g$. 
\end{remark}
\begin{remark}
\autoref{thm3.1new} provides us with the following surprising implication:
$$\lim_{n \rightarrow \infty} F^{n}\left(\frac{c_{Bon} \cdot d}{\sqrt{1-\rho}}\right)=1 \implies \lim_{n \rightarrow \infty} FWER_{Bon}(n, \alpha, \rho)=0.$$
\end{remark}
Under some extra conditions on $f$ and $g$, we prove $\displaystyle \lim_{n \rightarrow \infty} F^{n}\left(\frac{c_{Bon} \cdot d}{\sqrt{1-\rho}}\right)=1$. Towards this, we need the following lemma.
\begin{lemma}\label{lemma3.1new}
Consider the equicorrelated setup mentioned in Section 2.1 with $\alpha \in (0,1), \rho \in (0,1)$. We have, 
$$\log \left[\lim _{n \rightarrow \infty} F^{n}\left(\frac{c_{Bon}}{\sqrt{1-\rho}}\right)\right] = -\alpha \cdot\left[\lim _{n \rightarrow \infty} \frac{f\left(\frac{c_{Bon}}{\sqrt{1-\rho}}\right)}{\int_{-\infty}^{\infty} f\left(\frac{c_{Bon}-\sqrt{\rho} u}{\sqrt{1-\rho}}\right) g(u) d u}\right].$$

\end{lemma}
\begin{proof}[\textbf{\upshape Proof of Lemma \ref{lemma3.1new}.}]
\noindent Suppose $P_{n}=F^{n}\left(c_{Bon} \cdot t\right)$ where $t=1/\sqrt{1-\rho}$. We wish to find $\displaystyle \lim _{n \rightarrow \infty} P_{n}$. Now,
\begin{align*}
\log \left[\lim _{n \rightarrow \infty} P_{n}\right] & =\lim _{n \rightarrow \infty}\left(\log P_{n}\right) \\
& =\lim _{n \rightarrow \infty} \big[n \cdot \log \left(F\left(c_{Bon} \cdot t\right)\right)\big] \\
& =\lim _{n \rightarrow \infty} \frac{\log \left(F\left(c_{Bon}\cdot t\right)\right)}{\frac{1}{n}} \\
& =\lim _{n \rightarrow \infty} \frac{h_{1}(n)}{h_{2}(n)} \quad (\text{say}) \\
& =\lim _{n \rightarrow \infty} \frac{h_{1}^{\prime}(n)}{h_{2}^{\prime}(n)} \quad \left(\text{since $h_{1},h_{2} \rightarrow 0$ as $n \rightarrow \infty$}\right) \\
& =\lim _{n \rightarrow \infty} \frac{\frac{1}{F\left(t c_{Bon}\right)} \cdot f\left(t c_{Bon}\right) \cdot t \cdot \frac{d}{d n}\left(c_{Bon}\right)}{-\frac{1}{n^{2}}} \\
& =\lim _{n \rightarrow \infty}-n^{2} t \cdot \frac{f\left(t c_{Bon}\right)}{F\left(t c_{Bon}\right)} \cdot \frac{\alpha}{n^{2} \cdot f^{*}\left(c_{Bon}\right)} \quad \left(\text{since $f^{*}\left(c_{Bon}\right)\cdot \frac{d}{d n}\left(c_{Bon}\right) =\alpha/n^2$}\right)\\
& =\lim _{n \rightarrow \infty}-\alpha t \cdot \frac{f\left(t c_{Bon}\right)}{f^{*}\left(c_{Bon}\right)} \\
& =-\alpha t \cdot \lim _{n \rightarrow \infty} \frac{f\left(t c_{Bon}\right)}{f^{*}\left(c_{Bon}\right)}.
% & \Rightarrow f_{\alpha, n}=\left(F^{*}\right)^{-1}\left(1-\frac{\alpha}{n}\right) \\
% & \Rightarrow\left(c_{Bon}\right)=1-\frac{\alpha}{n} \\
% & \Rightarrow\left(a_{\alpha, n}\right) \cdot \frac{d}{d n}\left(c_{Bon}\right)=\frac{\alpha}{n^{2}}
\end{align*}
\noindent The rest follows once one observes that
$$
f^{*}(v)=\frac{1}{\sqrt{1-\rho}} \int_{-\infty}^{\infty} f\left(\frac{v-\sqrt{\rho} u}{\sqrt{1-\rho}}\right) g(u) d u.$$
\end{proof}
\noindent We are now in a position to state the foremost result of this paper:
\begin{theorem}\label{thm3.2new}
Consider the equicorrelated setup mentioned in Section 2.1 with $\alpha \in (0,1), \rho \in (0,1)$. Under the global null hypothesis, we have
$$\lim_{n \rightarrow 0} FWER_{Bon}(n, \alpha, \rho) = 0$$
if the following two conditions are satisfied:

\noindent (i) There exists positive constant $a$ such that $G(a)<1$. 

\noindent (ii) The density $f$ is non-increasing on $\mathbb{R}^{+}$ such that
$$\lim _{x \rightarrow \infty} \frac{f(x)}{f(x-b)}=0 \quad \forall b>0.$$
\end{theorem}

The conditions imposed on $f$ is satisfied by a wide range of densities, including the normal distribution, the generalized normal distributions with any shape parameter not less than 2, and a wide class of distributions from the exponential family.

\begin{proof}[\textbf{\upshape Proof of \autoref{thm3.2new}.}]
We have, 
$$
\begin{aligned}
& \int_{-\infty}^{\infty} f\left(\frac{c_{Bon}-\sqrt{\rho} u}{\sqrt{1-\rho}}\right) g(u) d u \\
\geqslant & \int_{a}^{\frac{C_{Bon}}{\sqrt{\rho}}} f\left(\frac{c_{Bon}-\sqrt{\rho} u}{\sqrt{1-\rho}}\right) g(u) d u \quad \text{(for all $a \geq 0$)}\\
\geqslant & \int_{a}^{\frac{C_{Bon}}{\sqrt{\rho}}} f\left(\frac{c_{Bon}-\sqrt{\rho} a}{\sqrt{1-\rho}}\right) g(u) d u \quad \text{(since $f$ is non-increasing on $\mathbb{R}^{+}$)}\\
= & f\left(\frac{c_{Bon}-\sqrt{\rho} a}{\sqrt{1-\rho}}\right) \cdot \bigg[G \left(\frac{c_{Bon}}{\sqrt{\rho}} \right)-G(a)\bigg].
\end{aligned}
$$
Under the first condition of \autoref{thm3.2new}, we have,
$$\lim_{n \rightarrow \infty} \frac{f\left(\frac{c_{Bon}}{\sqrt{1-\rho}}\right)}{\int_{-\infty}^{\infty} f\left(\frac{c_{Bon}-\sqrt{\rho} u}{\sqrt{1-\rho}}\right) g(u) d u} \leq \beta \cdot \lim_{n \to \infty} \frac{f\left(\frac{c_{Bon}}{\sqrt{1-\rho}}\right)}{f\left(\frac{c_{Bon}-\sqrt{\rho} a}{\sqrt{1-\rho}}\right)}$$
for some positive constant $\beta$. 

The quantity on the right side of the above inequality is zero from the second condition of \autoref{thm3.2new}. Hence, the quantity on the left is also zero. Then, invoking Lemma \ref{lemma3.1new}, we obtain \begin{equation}\displaystyle \lim_{n \rightarrow \infty} F^{n}\left(\frac{c_{Bon}}{\sqrt{1-\rho}}\right)=1, \quad \forall \rho \in (0,1).\label{limitcbon}\end{equation}

Let us now take $d = \sqrt{\frac{1-\rho}{1-\rho(1-\epsilon)}}$ for $\epsilon \in (0,1)$. Then, 
\begin{equation}\label{limitcbon2}\lim_{n \rightarrow \infty} F^{n}\left(\frac{c_{Bon} \cdot d}{\sqrt{1-\rho}}\right) = \lim_{n \rightarrow \infty} F^{n}\left(\frac{c_{Bon}}{\sqrt{1-\rho(1-\epsilon)}}\right) =1 \quad (\text{from \eqref{limitcbon}})\end{equation}

Taking the limit in \autoref{thm3.1new} and then applying equations \eqref{limitcbon} and \eqref{limitcbon2} in it, we obtain the desired limiting zero result. \end{proof}

\begin{remark}
Consider the special case that $U$ has support on $\mathbb{R}^{+}$. Then, $G(0)=0$ and \autoref{thm3.1new} gives us, for each $d \in (0,1)$,
$$
FWER_{Bon}(n, \alpha, \rho) \leqslant 1 - G\left(c_{Bon} \cdot \frac{1-d}{\sqrt{\rho}}\right) \cdot F^{n}\left(\frac{c_{Bon} \cdot d}{\sqrt{1-\rho}}\right)
$$
under the global null hypothesis. Hence, \autoref{thm3.2new} holds even when support of $U$ is a subset of $\mathbb{R}^{+}$. 
\end{remark}
\begin{remark}
In proving \autoref{thm3.2new}, we have actually showed that the upper bound on $FWER_{Bon}(n, \alpha, \rho)$ mentioned in \autoref{thm3.1new} converges to zero as $n \to \infty$ for any $d \in (\sqrt{1-\rho},1)$. Hence, \autoref{thm3.1new} is a much stronger result than the ones mentioned in \cite{deycstm}, since the upper bounds mentioned therein do not converge to zero as $n \to \infty$. 
\end{remark} 

We extend \autoref{thm3.2new} to arbitrary configurations of the null hypotheses:
\begin{theorem}\label{thm3.4new}
Consider the equicorrelated setup mentioned in Section 2.1 with $\alpha \in (0,1), \rho \in (0,1)$. Under any configuration of true and false null hypotheses, we have
$$\lim_{n \rightarrow 0} FWER_{Bon}(n, \alpha, \rho) = 0$$
if the following two conditions are satisfied:

\noindent (i) There exists positive constant $a$ such that $G(a)<1$. 

\noindent (ii) The density $f$ is non-increasing on $\mathbb{R}^{+}$ such that
$$\lim _{x \rightarrow \infty} \frac{f(x)}{f(x-b)}=0 \quad \forall b>0.$$
\end{theorem}
\begin{proof}[\textbf{\upshape Proof of \autoref{thm3.4new}}]
Without loss of generality, we may assume that the set of true null hypotheses $\mathcal{A}$ has cardinality $n_{0}$. Now, 
\begin{align*}
1-FWER_{Bon}(n, \alpha, \rho)=&\mathbb{P}\left(X_{i} \leqslant c_{Bon} \quad \forall i\in \mathcal{A}\right)\\
=& \mathbb{E}_{U}\left[F^{n_0}\left(\frac{c_{Bon}-\sqrt{\rho} U}{\sqrt{1-\rho}}\right)\right] \\
\geq & \mathbb{E}_{U}\left[F^{n}\left(\frac{c_{Bon}-\sqrt{\rho} U}{\sqrt{1-\rho}}\right)\right] \text{(since $0 \leqslant F(\cdot) \leqslant 1$ and $1 \leqslant n_{0} \leqslant n$)}\\
\longrightarrow & 1 \hspace{1.5mm} \text{as} \hspace{1.5mm} n\to \infty.
\end{align*}
The rest follows. \end{proof}

\subsection{Positive Limit of Bonferroni FWER}

Proceeding along the similar lines as in the proof of \autoref{thm3.1new}, we have
$$
\begin{aligned}
FWER_{Bon}(n, \alpha, \rho) & =1 -\int_{-\infty}^{\infty} F^{n}\left(\frac{c_{Bon}-\sqrt{\rho} u}{\sqrt{1-\rho}}\right) g(u) d u \\
& =1 -\int_{0}^{\infty} F^{n}\left(\frac{c_{Bon}-\sqrt{\rho} u}{\sqrt{1-\rho}}\right) g(u) d u-\int_{-\infty}^{0} F^{n}\left(\frac{c_{Bon}-\sqrt{\rho} u}{\sqrt{1-\rho}}\right) g(u) d u \\
& \geq 1 - F^{n}\left(\frac{c_{Bon}}{\sqrt{1-\rho}}\right) \cdot[1-G(0)]- G(0)\\
& = \bigg[1-F^{n}\left(\frac{c_{Bon}}{\sqrt{1-\rho}}\right)\bigg]\cdot [1-G(0)].
\end{aligned}$$
Hence, if we consider only those distributions $G$ for which $G(0)<1$, then, $\displaystyle \lim _{n \rightarrow \infty} F^{n}\left(\frac{c_{Bon}}{\sqrt{1-\rho}}\right)<1$ implies $\displaystyle \lim _{n \rightarrow \infty}FWER_{Bon}(n, \alpha, \rho)>0$. Therefore, our present goal (in this subsection) is to find distributions $F$ for which
$$
\lim _{n \rightarrow \infty} F^{n}\left(\frac{c_{Bon}}{\sqrt{1-p}}\right)<1.
$$
Towards this, we shall utilize \textit{Cantelli's Inequality}. This states that for any real-valued random variable $X$ with finite variance $\sigma^{2}$, 
$$\mathbb{P}[X\geqslant \lambda + E(X)] \leqslant \frac{\sigma^{2}}{\sigma^{2}+\lambda^{2}} \quad \forall \lambda>0.$$
We apply the above inequality on the test statistic $X_{i}:=\sqrt{1-\rho} Z_{i}+\sqrt{\rho} U$ with $\lambda=\sqrt{n/\alpha}$. Using Cantelli's inequality, we obtain
$$
\begin{aligned}
& \mathbb{P}[X_i \geqslant \sqrt{n/\alpha}+\mu_i] \leqslant \frac{\alpha}{n} \\
\implies & \mathbb{P}[X_i \leqslant \sqrt{n/\alpha}+\mu_i] \geqslant 1-\frac{\alpha}{n} \\
\implies & F^{\star}\left(\sqrt{n/\alpha}+\mu_i\right) \geqslant 1-\frac{\alpha}{n} \\
\implies & \sqrt{n/\alpha}+\mu_i \geqslant c_{Bon}\\
\implies & \sqrt{n/\gamma} \geqslant c_{Bon} \quad \text{(for some $\gamma \in (0,\alpha)$)}.
\end{aligned}$$
Therefore, 
$$
\begin{aligned}
FWER_{Bon}(n, \alpha, \rho) 
& \geq \bigg[1-F^{n}\left(\frac{c_{Bon}}{\sqrt{1-\rho}}\right)\bigg]\cdot [1-G(0)] \\
& \geq \bigg[1-F^{n}\left(\sqrt{\frac{n}{\gamma(1-\rho)}}\right)\bigg]\cdot [1-G(0)].
\end{aligned}$$
Hence, it is enough to find a cdf $F$ for which $\displaystyle \lim_{n \to \infty} F^{n}\left(\sqrt{\frac{n}{\gamma(1-\rho)}}\right) < 1$. Towards this, we consider the density of the Pareto distribution:
$$
f_{Z}(z)=\left\{\begin{array}{cl}
\frac{b \eta^{b}}{z^{b+1}} & z \geqslant \eta \\
0 & \text{else}.
\end{array}\right.
$$
Take $b=2+\delta$ and $\eta=\sqrt{\frac{\delta}{2+\delta}} \cdot(\delta+1)$ with $\delta>0$. Evidently, $Z$ has variance 1. The cdf of $Z$ is given by
$$
\begin{aligned}
& F_{Z}(z)=\left\{\begin{array}{cc}
1-\eta^{b} \cdot z^{-b} & z \geq \eta \\
0 & \text{else}.
\end{array}\right. \\
\end{aligned}
$$
This implies
$$
\begin{aligned}
F^{n}\left(\sqrt{\frac{n}{\gamma(1-\rho)}}\right) & =\left(1-\frac{d}{n^{1+\frac{\delta}{2}}}\right)^{n} \quad (\text{here $d:=\eta^{1+\delta/2}$ is a constant})\\
& \geqslant\left(1-\frac{d}{n^{1+\frac{\delta}{2}}}\right)^{n^{1+\frac{\delta}{2}}} \displaystyle \underset{n \to \infty}{\longrightarrow} e^{-d}.
\end{aligned}
$$

\noindent Hence, in this case, the limiting Bonferroni FWER is strictly positive.

\noindent We may summarize the findings of this subsection through the following result: 

\begin{theorem}\label{thm3.3new}
Consider the equicorrelated setup mentioned in Section 2.1 with $\alpha \in (0,1), \rho \in (0,1)$. Under the global null hypothesis, we have
$$\lim_{n \rightarrow 0} FWER_{Bon}(n, \alpha, \rho) > 0$$
if the following two conditions are satisfied:

\noindent (i) $G(0)<1$.

\noindent (ii)$\displaystyle \lim_{n \to \infty} F^{n}\left(\sqrt{\frac{n}{\gamma(1-\rho)}}\right) < 1$ for some $\gamma$ satisfying $\sqrt{n/\gamma} \geqslant c_{Bon}$. 
\end{theorem}

%It is also mention-worthy that most of the existing literature on the performance of FWER procedures under dependence consider the equicorrelated case. 

%It is mention-worthy that FWER is not the probability of making any false rejections when all null hypotheses are true. In order to control FWER at $\alpha$ level, we need to ensure the probability of making any false rejection be less than $\alpha$ under any configuration of true and false null hypotheses. However, for the sake of technical simplicity we shall compute the probability in r.h.s of (1) under the global null $H_0$ at first (and consider that as the definition of FWER) and then extend the results obtained in this case to any configuration of true and false null hypotheses. 
%\newpage

%. Suppose we are interested in testing $m \geq 2$ hypotheses $H_1, \ldots, H_m$ and let $p_1, \ldots, p_m$ denote $p$-values for testing them. In the manyone situation hypotheses of interest are, for example, $H_i: \vartheta_i-\vartheta_0 \leq 0$ versus $K_i: \vartheta_i-\vartheta_0>0$ or the corresponding two-sided hypotheses, where $\vartheta_i$ denotes the effect of the $i$ th treatment for $i=1, \ldots, m$ and $\vartheta_0$ denotes a placebo or standard treatment effect. 

\section{Limiting Bonferroni FWER under General Positive Correlation\label{sec:4}}

A plethora of scientific avenues involves variables with complex correlation structures. This section makes an attempt to generalize the theoretical limiting zero results for the Bonferroni FWER under those circumstances. Towards this, we consider the same \textit{sequence model} setup mentioned in Section 2.1, but now  $\operatorname{\mathbb{C}orr}\left(X_{i}, X_{j}\right)=\rho_{ij}$ for $i \neq j$ where $\rho_{ij} \in (0,1]$. Let $\Sigma_{n}$ be the correlation matrix of $X_{1}, \ldots, X_{n}$. We also assume that 
the joint density $h(X_1, \ldots, X_n)$ is elliptically contoured \citep{SomeshD}, i.e., it is given by 

$$h(\mathbf{x}) = |\Sigma_n|^{-1 / 2} t\left(\mathbf{x}^{\prime} {\Sigma_{n}}^{-1} \mathbf{x}\right)$$
where $\int_0^{\infty} r^{n-1} t\left(r^2\right) d r<\infty$. 

\noindent The following result extends \autoref{thm3.1new} to this setup. 

\begin{theorem}\label{section4thm1} Suppose the joint density of $(X_{1}, \ldots, X_{n})$ is elliptically contoured. Moreover, let $\liminf \rho_{ij}=\delta>0$. Then, for any $\alpha \in (0,1)$,
$$\lim_{n \to \infty}FWER_{Bon}(n,\alpha,\Sigma_{n}) = 0$$
under any configuration of true and false null hypotheses, if the following two conditions are satisfied:

\noindent (i) There exists positive constant $a$ such that $G(a)<1$. 

\noindent (ii) The density $f$ is non-increasing on $\mathbb{R}^{+}$ such that
$$\lim _{x \rightarrow \infty} \frac{f(x)}{f(x-b)}=0 \quad \forall b>0.$$
\end{theorem}

\noindent The following famous inequality due to \cite{SomeshD} will be useful towards proving this.

\begin{theorem}\label{Somesh}
Suppose the joint density of $(X_{1}, \ldots, X_{k})$ is elliptically contoured. Consider the quadrant probability
$$l(k, \mathbf{a},\Sigma)=\mathbb{P}_{\Sigma}\left[\bigcap_{i=1}^{k}\left\{X_{i} \leqslant a_{i}\right\}\right].$$ Let $\mathbf{R}=\left(\rho_{i j}\right)$ and $\mathbf{T}=\left(\tau_{i j}\right)$ be two correlation matrices. If $\rho_{i j} \geqslant \tau_{i j}$ holds for all $i, j$, then $l(k, \mathbf{a}, \mathbf{R}) \geq l(k, \mathbf{a}, \mathbf{T})$, i.e
$$
\mathbb{P}_{\Sigma=\mathbf{R}}\left[\bigcap_{i=1}^{k}\left\{X_{i} \leqslant a_{i}\right\}\right] \geqslant \mathbb{P}_{\Sigma=\mathbf{T}}\left[\bigcap_{i=1}^{k}\left\{X_{i} \leqslant a_{i}\right\}\right]
$$
holds for all $\mathbf{a}=\left(a_{1}, \ldots, a_{k}\right)^{\prime}$. 
\end{theorem} 

\begin{proof}[\textbf{\upshape Proof of \autoref{section4thm1}}]
This proof is quite similar to the proof of Theorem 4.3.4 of \cite{DeyThesis}. We prove \autoref{section4thm1} only under the global null since the proof under general configuration can be considered similarly as the preceding proof. 
For fixed $n \in \mathbb{N}$, suppose 
$$\mathcal{M}_n := \bigg\{i \in \{1, \ldots, n\} : \forall j\neq i, \rho_{ij} \geq \delta \bigg\}.$$
Now, 
\begin{align*}
& FWER_{Bon}(n,\alpha,\mathbf{\Sigma}) \\
= &
\mathbb{P}_{\Sigma_n}\bigg(\bigcup_{i=1}^{n}\{X_i > c_{Bon}\}\bigg) \\
= & \mathbb{P}_{\Sigma_n}\bigg(\bigcup_{i\in \mathcal{M}_n}\{X_i > c_{Bon}\}\bigcup \bigcup_{i\notin \mathcal{M}_n}\{X_i > c_{Bon}\}\bigg) \\
\leq & \mathbb{P}\bigg(\bigcup_{i\in \mathcal{M}_n}\{X_i > c_{Bon}\}\bigg)  + \mathbb{P}\bigg( \bigcup_{i\notin \mathcal{M}_n}\{X_i > c_{Bon}\}\bigg) \quad (\text{using Boole's inequality})\\
\leq & \mathbb{P}\bigg(\bigcup_{i\in \mathcal{M}_n}\{X_i > c_{Bon}\}\bigg) + [n-|\mathcal{M}_n|] \cdot \frac{\alpha}{n} \quad (\text{using Boole's inequality})\\
\leq & \mathbb{P}\bigg(\bigcup_{i\in \mathcal{M}_n}\{X_i > c_{Bon}^{\prime}\}\bigg) + [n-|\mathcal{M}_n|]  \cdot \frac{\alpha}{n} \quad (\text{where $c_{Bon}^{\prime}=(F^{\star})^{-1}(1-\alpha/|\mathcal{M}_n|)\leq c_{Bon}$})\\
= & 1-l(|\mathcal{M}_n|, \mathbf{a}, \mathbf{\Sigma}_{\mathcal{M}_n}) + [n-|\mathcal{M}_n|] \cdot \frac{\alpha}{n}
\end{align*}
where $l$ is as defined in \autoref{Somesh}, $\mathbf{\Sigma}_{\mathcal{M}_n}$ is the covariance matrix of $(X_{i} : i \in \mathcal{M}_n)$, and $a_{i}=c_{Bon}^{\prime}$ for $i \in \mathcal{M}_n$. 

\autoref{Somesh} gives  $l(|\mathcal{M}_n|, \mathbf{a},\Sigma_{\mathcal{M}_n}) \geq l(|\mathcal{M}_n|, \mathbf{a}, \Gamma_{|\mathcal{M}_n|}(\delta))$. So, 
\begin{align*}
 FWER_{Bon}(n,\alpha,\Sigma_{n}) & \leq 1 - l(|\mathcal{M}_n|, \mathbf{a}, \Gamma_{|\mathcal{M}_n|}(\delta))
 + [n-|\mathcal{M}_n|] \cdot \frac{\alpha}{n} \\
 & = FWER_{Bon}(|\mathcal{M}_n|,\alpha, \delta)
 + [n-|\mathcal{M}_n|] \cdot \frac{\alpha}{n}.
\end{align*}
Since $\liminf \rho_{ij}=\delta>0$, we also have $n-|\mathcal{M}_n|$ is finite. The rest follows from \autoref{thm3.4new} by taking $n \to \infty$. 
\end{proof}

\section{Limiting FWER of Step-down Procedures\label{sec:5}} 

Stepdown MTPs \citep{BL, BL2, Holm1979} usually enjoy better power than the single-step procedures, while still controlling FWER at $\alpha$. A $p$-value based step-down MTP utilizes a vector of cutoffs $\textbf{u} =\left(u_1, \ldots, u_n\right)$, where $0 \leq u_1 \leq \ldots \leq u_n \leq 1$. Let $M=\max \left\{i: P_{(j)} \leq u_j\right.$ for all $\left.j=1, \ldots, i\right\}$. Then the step-down procedure based on critical values $\textbf{u}$ rejects $H_{(01)}, \ldots, H_{(0M)}$. 

% \begin{example} 
% The Bonferroni method is a step-down procedure with $u_i=\alpha / n$, $i=1, \ldots, n$.
% \end{example}

% \begin{example} 
% The Sidak method is a step-down MTP with $u_i=1 - (1-\alpha)^{1 / n}$, $i=1, \ldots, n$.
% \end{example}

% \begin{example} 
% The Holm \citep{Holm1979} method is a popular step-down MTP with $u_i=\alpha /(n-i+1), i=1, \ldots, n$.
% \end{example}

% \begin{example} 
% \cite{BL} introduced a step-down MTP with
% $$u_i=\min \left(1, \frac{n q}{(n-i+1)^2}\right), \quad 1 \leq i \leq n \quad(0<q<1).$$
% \end{example}

% \begin{example} 
% \cite{BL2} studied another step-down MTP with
% $$
% u_i=1 - \Bigg[1 - \min \left(1, \frac{n q}{n-i+1}\right)\Bigg]^{1/{(n-i+1)}}, \quad 1 \leq i \leq n \quad(0<q<1).
% $$
% \end{example}

% \begin{example} 
% \cite{BL2} mentioned a Holm-type procedure with the critical values
% $$
% u_i=1 - (1 - q)^{1/{(n-i+1)}}, \quad 1 \leq i \leq n \quad(0<q<1).
% $$
% \end{example}

Holm method \citep{Holm1979} is a popular step-down MTP which uses $u_i=\alpha /(n-i+1), i=1, \ldots, n$. It controls the FWER under any dependence of the test statistics. \cite{deybhandari} showed the following on the limiting behavior of Holm FWER: 

\begin{theorem}\label{thm5.1} \citep{deybhandari}
Suppose $\mu^{\star} = \sup \mu_{i} < \infty$. Then, under any configuration of true and false null hypotheses, we have
$$\displaystyle \lim_{n \to \infty} FWER_{Holm}(n, \alpha, \rho) = 0 \quad \text{for all} \hspace{2mm} \alpha \in (0,1) \hspace{2mm} \text{and} \hspace{2mm} \rho \in (0,1].$$
\end{theorem}
\cite{deystpa} extends this result to any step-down MTP under some non-negatively correlated normal setups:

\begin{theorem}\label{thm5.2}\citep{deystpa}
Let $\Sigma_n$ be the correlation matrix of $X_1, \ldots, X_n$ with $(i,j)$’th entry $\rho_{ij}$ such that $\liminf \rho_{ij}=\delta>0$. Suppose $\sup \mu_{i}$ is finite and $T$ is any step-down MTP controlling FWER at level $\alpha \in (0,1)$. Then, for any $\alpha \in (0,1)$,
$$\lim_{n \to \infty}\mathbb{P}_{\Sigma_{n}}\bigg(R_{n}(T) \geq 1\bigg) = 0$$
under any configuration of true and false null hypotheses. Consequently, 
$$\lim_{n \to \infty}FWER_{T}(n,\alpha,\Sigma_{n}) = 0 \quad and \quad \lim_{n \to \infty} AnyPwr_{T}(n,\alpha,\Sigma_{n}) = 0.$$
\end{theorem}

Here we shall extend \autoref{thm5.2} to general positively correlated elliptically contoured setups. 

\begin{theorem}\label{thm5.3}
Suppose the joint density of $(X_{1}, \ldots, X_{n})$ is elliptically contoured. Moreover, let $\liminf \rho_{ij}=\delta>0$. Suppose $\sup \mu_{i}$ is finite and $T$ is any step-down MTP controlling FWER at level $\alpha \in (0,1)$. Then, for any $\alpha \in (0,1)$,
$$\lim_{n \to \infty}\mathbb{P}_{\Sigma_{n}}\bigg(R_{n}(T) \geq 1\bigg) = 0$$
under any configuration of true and false null hypotheses, if the following two conditions are satisfied:

\noindent (i) There exists positive constant $a$ such that $G(a)<1$. 

\noindent (ii) The density $f$ is non-increasing on $\mathbb{R}^{+}$ such that
$$\lim _{x \rightarrow \infty} \frac{f(x)}{f(x-b)}=0 \quad \forall b>0.$$

\end{theorem}

Theorem \ref{thm5.3} is the most general result in this work result since it works under a quite general class of elliptically contoured distributions, covers a wide class of correlation structures, encompasses all step-down FWER controlling procedures and also accommodates any configuration of true and false null hypotheses. We establish Theorem \ref{thm5.3} using the following result:

\begin{theorem}\label{thm5.4} \citep{Gordon2008}
Let $T$ be a step-down MTP based on the set of cut-offs $\mathbf{u} \in \mathcal{S}^n$. If $FWER_{T} \leq \alpha<1$, then $u_i \leq \alpha /(n-i+1), i=1, \ldots, n$.
\end{theorem}

\begin{proof}[\textbf{\upshape Proof of \autoref{thm5.3}.}]
We have,
\begin{align*}
 & \mathbb{P}_{\Sigma_{n}}(R_{n}(Holm)\geq 1) \\
 = & \mathbb{P}_{\Sigma_{n}}(P_{(1)} \leq \alpha/n) \\
 = & \mathbb{P}_{\Sigma_{n}}(X_{(n)} \geq c_{Bon}) \\
\leq & \mathbb{P}_{\Gamma_{n}(\delta)}(X_{(n)} \geq c_{Bon}) \quad \text{(using \autoref{Somesh})}\\
= & 1 - \mathbb{P}_{\Gamma_{n}(\delta)}(X_{(n)} \leq c_{Bon})\\
=&1 - \mathbb{P}_{\Gamma_{n}(\delta)}\left(X_{i} \leqslant c_{Bon} \quad \forall i=1,2, \ldots, n\right)\\
=&1 - \mathbb{P}_{\Gamma_{n}(\delta)}\Bigg[\bigcap_{i\in \mathcal{A}} \{X_i \leqslant c_{Bon} \} \bigcap \bigcap_{i \notin \mathcal{A}} \{X_i \leqslant c_{Bon} \} \Bigg]\\
=&1 - \mathbb{E}_{U}\Bigg[\bigg\{\prod_{i \in \mathcal{A}} F \left(\frac{c_{Bon}-\sqrt{\rho}U - \mu_{i}}{\sqrt{1-\rho}}\right)\bigg\}\cdot F^{n-|\mathcal{A}|}\left(\frac{c_{Bon}-\sqrt{\rho}U }{\sqrt{1-\rho}}\right)\Bigg]\\
\leq &1 - \mathbb{E}_{U} \left[F^{n}\left(\frac{c_{Bon}-\sqrt{\rho}U  - \mu^{\star}}{\sqrt{1-\rho}}\right)\right].
\end{align*}
The last quantity above tends to zero asymptotically since $\mu^{\star} < \infty$ (its proof is exactly similar to that of Theorem 2 of \cite{deybhandari}. Hence, 
\begin{equation}\label{holm2}
    \lim_{n \to \infty} \mathbb{P}_{\Sigma_{n}}(R_{n}(Holm)\geq 1) = 0.
\end{equation}
\noindent We observe,
$$\begin{aligned}
\mathbb{P}_{\Sigma_{n}}(R_{n}(T)\geq 1)
= \mathbb{P}_{\Sigma_{n}}(P_{(1)} \leq u_{1})
\leq & \mathbb{P}_{\Sigma_{n}}(P_{(1)} \leq \alpha/n) = \mathbb{P}_{\Sigma_{n}}(R_{n}(Holm)\geq 1).
\end{aligned}$$
\noindent The inequality above holds since we have $u_1 \leq \alpha/n$ from Theorem \ref{thm5.4}. The rest follows from \eqref{holm2}.
\end{proof}

\section{Empirical study and Discussion\label{sec:6}}
Under the general equicorrelated sequence model, we have $$X_{i} = \sqrt{1-\rho}Z_{i} + \sqrt{\rho} U,  \quad i \in \{1, \ldots,n\},$$
where $Z_{i} \sim F(0,1)$ (under the global null) and $U \sim G(0,1)$ independently. Also, $Z_{i}$’s are independent to each other. In this framework, one has
\begin{align*}
 FWER_{Bon}(n, \alpha, \rho)  = 1 - \int_{-\infty}^{\infty} F^{n}\left(\frac{c_{Bon}-\sqrt{\rho} u}{\sqrt{1-\rho}}\right) g(u) d u.
\end{align*}

To illustrate the theoretical findings of this work, we include the actual values of FWER in some specific examples (i.e, for some particular choices of $F$ and $G$). 

Table \ref{normal_laplace} presents the values of $FWER_{Bon}(n,\alpha=.05,\rho)$ when $Z_{i}\sim N(0,1)$ and $U\sim \operatorname{Laplace}(\text{mean}=0, \text{scale}=1/\sqrt{2})$. One notes that, for each positive $\rho$, FWER values decrease as $n$ grows. On the other hand, for each $n$, the FWER values decrease as $\rho$ increases. Also, the rate of decay is much faster for larger values of $\rho$. 
\newpage
\begin{table*}[!h]
\caption{$FWER_{Bon}(n,\alpha=.05,\rho)$ values for $Z_{i}\sim N(0,1)$ and $U\sim \operatorname{Laplace}(\text{mean}=0, \text{scale}=1/\sqrt{2})$}
\label{normal_laplace}
\begin{center}
\begin{tabular}{@{}ccccccc@{}}
\hline
Correlation & \multicolumn{6}{c}{Number of hypotheses ($n$)}\\
\cline{2-7}
\hspace{4.7mm} $(\rho)$ & \multicolumn{1}{c}{$100$}
& \multicolumn{1}{c}{$1000$} & \multicolumn{1}{c}{$10000$}
& \multicolumn{1}{c}{$100000$} & \multicolumn{1}{c}{$1$ Million} & \multicolumn{1}{c}{$10$ Million}\\
\hline
0    & 0.048782 &  0.048771 & 0.048771 & 0.048767 &  0.048737& 0.048381\\
0.1  & 0.042157 & 0.0358767 & 0.028132 & 0.020027 &  0.012776 & 0.007226\\
0.2  & 0.029444 & 0.0170793 & 0.008132 & 0.003311 & 0.0011920 & 0.000376 \\
0.3  & 0.018757 & 0.0076784 & 0.002604 & 0.000784 & 0.0002137 & 0.000021\\
0.4  & 0.011983 & 0.0038237 & 0.000936 & 0.000256 & 0.0000234 & 0.000013 \\
0.5  & 0.007902 & 0.0020776 & 0.000370 & 0.000101 & 0.0000133 & 0.000007 \\
0.6  & 0.005334 & 0.0011888 & 0.000151 & 0.000014 & 0.0000133 & 0.0000006 \\
0.7  & 0.003621 & 0.0006940 & 0.000017 & 0.000013 & 0.0000024 & 0.00000007\\
0.8  & 0.002410 & 0.0003983 & 0.000013 & 0.000013 & 0.00000007 & 0.00000007\\
0.9  & 0.001491 & 0.0002729 & 0.000013 & 0.000013 & 0.00000007 & 0.00000007 \\
\hline
\end{tabular}
\end{center}
\end{table*}

Table \ref{normal_t4} portrays the FWER values when $Z_{i}\sim N(0,1)$ and $\sqrt{2}U \sim t_{4}$. Here also, for each positive $\rho$, FWER values decrease as $n$ grows. However, it is striking to observe that the effect of $\rho$ on the FWER values is negligible in this case for large $n$. 

\begin{table*}[!h]
\caption{$FWER_{Bon}(n,\alpha=.05,\rho)$ values for $Z_{i}\sim N(0,1)$ and $\sqrt{2}U\sim t_4$}
\label{normal_t4}
\begin{center}
\begin{tabular}{@{}ccccccc@{}}
\hline
Correlation & \multicolumn{6}{c}{Number of hypotheses ($n$)}\\
\cline{2-7}
\hspace{4.7mm} $(\rho)$ & \multicolumn{1}{c}{$100$}
& \multicolumn{1}{c}{$1000$} & \multicolumn{1}{c}{$10000$}
& \multicolumn{1}{c}{$100000$} & \multicolumn{1}{c}{$1$ Million} & \multicolumn{1}{c}{$10$ Million}\\
\hline
0   & 0.048782 & 0.048772 & 0.048771 & 0.048771 & 0.048776 & 0.048770 \\
0.1 &0.033862 &0.009390 &0.000169 &0.000004 &0.000000104 & 0.000000104 \\
0.2 & 0.016563 &0.000938 &0.000062 &0.000003681 &0.000000104 & 0.000000104 \\
0.3 & 0.007615 &0.000399 &0.000044 &0.000003681 &0.000000104 & 0.000000104 \\
0.4 & 0.004184 &0.000249 &0.000036 &0.000003681 &0.000000104 & 0.000000104 \\
0.5 & 0.002821 &0.000179 &0.000031 &0.000003681 &0.000000104 & 0.000000104 \\
0.6 & 0.001973 &0.000139 &0.000027 & 0.000003681 &0.000000104 & 0.000000104 \\
0.7 & 0.001499 &0.000111 &0.000025 &0.000003681 &0.000000104 & 0.000000104 \\
0.8 & 0.001129 &0.000092 &0.000022 &0.000003681 &0.000000104 & 0.000000104 \\
0.9 & 0.000856 &0.000075 &0.000020 &0.000003681 &0.000000104 & 0.000000104\\
\hline
\end{tabular}
\end{center}
\end{table*}

\section{Concluding Remarks\label{sec:7}}
FWER has been a linchpin of simultaneous statistical inference because of its simplicity and relevance across various disciplines. The Bonferroni method, one of the earliest MTPs, has naturally enjoyed widespread use in countless applications. Besides its innumerable applicability, it enjoys interesting theory that often illuminates surprising results and, at the same time, complements applied implications. 

For example, the asymptotic convexity of the Bonferroni FWER under the equicorrelated Normal setup \citep{das_2021} is a purely probabilistic result in multivariate Normal theory. However, it also sheds light on the extent of the conservativeness of Bonferroni. A series of works has culminated since then which has tried to sharpen or extend this result in different directions: 
\begin{itemize}
    \item to more general correlation structures \citep{deybhandari},
    \item to arbitrary configurations of the null hypotheses \citep{deybhandari},
    \item or general stepwise procedures \citep{deystpa}. 
\end{itemize} However, a unifying theory for a more general class of distributions is nonexistent in the present literature. 

In a related but different direction, a few other works on the non-asymptotic considerations under the correlated Gaussian sequence model framework \citep{deycstm, deybhandaristpa}. However, the non-asymptotic upper bounds proposed in \cite{deycstm} are not asymptotically sharp since they do not converge to zero as the number of hypotheses grows indefinitely. 

Thus, a unifying theory for a general class of distributions that encompasses both the asymptotic and non-asymptotic regimes is called for. This paper achieves this by obtaining upper bounds that are asymptotically sharp. We start with an upper bound on the family-wise error rate(FWER) of Bonferroni's procedure for equicorrelated test statistics. Consequently, we find that for a quite general class of distributions, Bonferroni FWER asymptotically tends to zero when the number of hypotheses approaches infinity. We extend this result to general positively correlated elliptically contoured setups and also to the class of step-down procedures. We also present an example of distributions for which Bonferroni FWER has a strictly positive limit under equicorrelation. \autoref{thm5.3} is kind of a \textit{mother theorem} in asymptotic FWER theory, since it works under a quite general class of elliptically contoured distributions, covers a wide class of correlation structures, encompasses all step-down FWER controlling procedures and also accommodates any configuration of true and false null hypotheses.

Lastly, there are possible scopes of interesting extensions in a few other directions. One potential extension is to study the (limiting) behaviors of Hochberg, Hommel, and Benjamini-Hochberg procedures under general distributions. Another natural related question is whether similar asymptotic results hold for other classes of MTPs, e.g., the class of consonant procedures \citep{WTRWH}. 

%Indeed, \cite{FDR2007} remark that it is challenging to deal with false discoveries in models with complicated dependence structures, e.g., in a multivariate Gaussian model with a general covariance matrix. 
%\section{Conclusions\label{sec:4}}
%Here are our conclusions.

%\newpage

%during the preparation of this manuscript. 

% \section*{Funding}
% This research did not receive any specific grant from funding agencies in the public, commercial, or not-for-profit sectors.
\vspace{8mm}

\noindent \textbf{\Large Acknowledgements.} 
\vspace{3mm}

\noindent The author expresses his heartfelt thanks to Prof. Subir Kumar Bhandari and Prof. Thorsten Dickhaus for their encouragement and insightful discussions during this work.
\vspace{6mm}

\noindent \textbf{\large Declarations of interest:} none.

\bibliography{references}
%\nocite{*}

\end{document}